\documentclass[a4paper,11pt]{article}

\usepackage[margin=1.05in]{geometry}
\usepackage{amsmath,amssymb,amsthm,mathtools}
\numberwithin{equation}{section}
\usepackage{enumitem}
\usepackage{hyperref}
\usepackage{microtype}
\usepackage{mathrsfs}
\hypersetup{colorlinks=true, linkcolor=blue, citecolor=blue, urlcolor=blue}

\newtheorem{theorem}{Theorem}[section]
\newtheorem{proposition}[theorem]{Proposition}
\newtheorem{lemma}[theorem]{Lemma}
\newtheorem{corollary}[theorem]{Corollary}
\newtheorem{definition}[theorem]{Definition}
\newtheorem{remark}[theorem]{Remark}

\newcommand{\R}{\mathbb{R}}
\newcommand{\Ent}{\operatorname{Ent}}

\newcommand{\id}{\operatorname{id}}
\newcommand{\supp}{\operatorname{spt}}
\newcommand{\ri}{\operatorname{ri}}

\newcommand{\diver}{\operatorname{div}}

\newcommand{\argmin}{\operatorname*{argmin}}
\newcommand{\argmax}{\operatorname*{argmax}}
\newcommand{\dd}{\mathrm d}
\newcommand{\D}{\mathrm D}
\newcommand{\cL}{\mathcal L}
\newcommand{\cP}{\mathcal P}

\newcommand{\mm}{\mathfrak m}

\title{\textbf{Wasserstein Barycenter Convexity Detects Hilbertian Geometry}}
\author{Bang-Xian Han\thanks{School of Mathematics, Shandong University, Jinan, China. Email: hanbx@sdu.edu.cn.}
	\and Deng-Yu Liu
	\thanks{School of Mathematical Sciences, University of Science and Technology of China, Hefei, China. Email: yzldy@mail.ustc.edu.cn}}

\date{\today}

\begin{document}
	\maketitle

		\begin{abstract}
We prove that convexity of the Boltzmann entropy at Wasserstein barycenters
is strong enough to distinguish Hilbert spaces from general Banach
spaces. Thus Wasserstein barycenters provide an intrinsic optimal-transport
test for Hilbertian geometry. More precisely, we show that if a finite-dimensional normed vector space, equipped with Lebesgue measure, satisfies the Wasserstein Jensen's inequality for the entropy at barycenters of arbitrary finite families of probability measures, then its norm must be induced by an inner product.

This contrasts sharply with a well-known result: every finite-dimensional normed vector space
satisfies the  nonnegative Ricci curvature condition in the  sense of Lott--Sturm--Villani, whereas barycenter convexity excludes all
non-Hilbertian norms. As a consequence, smooth reversible Finsler manifolds
satisfying the corresponding barycentric curvature-dimension condition have
Riemannian tangent norms.

The proof does not assume 
strict convexity of the norm. Its two main ingredients are a rank-one
polarization argument, which yields the dual parallelogram identity in the
strictly convex case, and a maximal-face trapping argument, which rules out
flat faces of the unit ball.
\end{abstract}

	

\tableofcontents

	\section{Introduction}\label{sec:introduction}
	The displacement convexity of entropy functionals along $W_2$ geodesics, introduced by McCann in \cite{McCann1997}, is an important theme in optimal transport. It relates optimal transport, lower Ricci curvature, and geometric functional inequalities. This point of view led Lott--Villani and Sturm to the synthetic curvature-dimension theory on metric-measure spaces \cite{LottVillani,SturmI,SturmII}. The relation between curvature bounds, optimal transport, and Pr\'ekopa--Leindler type inequalities goes back to the Riemannian interpolation inequalities of Cordero-Erausquin--McCann--Schmuckenschl\"ager \cite{CEMS,CEMS2006}.
	In Lott--Sturm--Villani's theory, curvature-dimension conditions are expressed through convexity of entropy along $L^2$-Wasserstein geodesics. This provides a formulation of lower Ricci bounds outside the smooth Riemannian setting.
	
	As Villani notes in \emph{``the last theorem in Old and New''} \cite[p.~926]{Villani2009}, every finite-dimensional normed vector space equipped with Lebesgue measure satisfies the standard synthetic non-negative Ricci condition in the weak Lott--Villani--Sturm sense. Thus  Lott--Villani--Sturm's $\mathrm{CD}$ condition  does not distinguish a Banach space from a Hilbert space. This is one reason for Ambrosio--Gigli--Savar\'e's \(\mathrm{RCD}\) theory \cite{AGS14,Gigli}. In that theory, Finsler structures are excluded by extra Riemannian assumptions: infinitesimal Hilbertianity, linearity of the heat flow, or equivalently the quadraticity of the Cheeger energy. For the Finsler side of the curvature-dimension theory, including weighted Ricci curvature and nonlinear analytic aspects, see also \cite{Ohta,Ohta2014,Ohta2017Survey,OhtaSturm2014}.
	
	\medskip
	The basic question of this paper is whether one can detect the
Riemannian, rather than merely Finsler, nature of a space using only
entropy inequalities in optimal transport.
For ordinary displacement convexity, the answer is negative: in the weak
Lott--Villani--Sturm theory, finite-dimensional Banach spaces already satisfy
the nonnegative curvature-dimension condition. Our result shows that the
answer becomes positive once geodesic interpolation is replaced by
Wasserstein barycenters.

	Motivated by Wasserstein barycenter geometry \cite{AguehCarlier,KimPass}, the barycenter curvature-dimension condition \(\mathrm{BCD}\) was introduced in \cite{HanLiuZhuBCD,HanLiuZhuSurvey}. In this condition displacement interpolation in the Lott--Sturm--Villani theory is replaced by Wasserstein barycenters, and the entropy functionals are required to satisfy the Wasserstein Jensen's inequality at Wasserstein barycenters of finite weighted families of measures.

	\begin{definition}[BCD condition]
		\label{def:bcd-infty}
		Let $(X,\mathsf{d},\mathfrak{m})$ be a geodesic metric measure space. We say that it satisfies $\mathrm{BCD}(K,\infty)$ if for every finite family \(\mu_i\in\cP_2(X)\) and every set of weights $\lambda_i>0$, $\sum_i\lambda_i=1$, there exists $\bar\mu \in\argmin_{\eta\in \cP_2(X)} \sum_{i}\lambda_i W_2^2(\eta,\mu_i)$ (called a Wasserstein barycenter) such that
		\begin{equation}\label{jensen}
			\Ent_\mm(\bar\mu)\le \sum_i\lambda_i\Ent_\mm(\mu_i)-\frac{K}{2} \sum_{i}\lambda_i W_2^2(\bar{\mu},\mu_i),
		\end{equation}
			where the Boltzmann entropy ${\rm Ent}_\mm(\cdot)$ is defined by
			\begin{equation*}
				{\rm Ent}_\mm(\mu):=
				\left \{\begin{array}{ll}
					\int \rho\ln \rho\,\dd\mm &\text{if } \mu\ll \mm,\ \mu=\rho\,\mm,\\
					+\infty &\text{otherwise}.
				\end{array}\right.
			\end{equation*}
			and $W_2(\cdot,\cdot)$
		is the 2-Wasserstein distance associated with $\mathsf{d}^2$. 
		
	\end{definition}

	For two marginals, weighted Wasserstein barycenters are the corresponding points on \(W_2\)-geodesics. Thus the above condition gives
	\[
	\mathrm{BCD}(0,\infty)\Longrightarrow \mathrm{CD}(0,\infty)
	\]
	in the weak Lott--Villani sense, with the same entropy convention.

	On smooth Riemannian manifolds, this stronger condition has the expected meaning. In the smooth Riemannian setting, the implication from a Ricci lower bound to the Wasserstein Jensen's inequality follows from the work of Kim--Pass \cite{KimPass}. Conversely, applying the two-marginal case recovers the usual displacement convexity of entropy, and the theorem of von Renesse--Sturm \cite{vRenesseSturm} gives the Ricci lower bound.
	
	\medskip
	
	The main theorem below states that, outside the Riemannian category, BCD has a rigidity that CD does not have. In particular, the barycenter curvature-dimension condition excludes the Finsler examples.  A vivid analogy:
	\begin{center}
	\emph{``Finsler geometry may masquerade as Riemannian geometry along a single geodesic, but it inevitably betrays its non-quadratic nature at the barycenter under multi-directional tension."}
	\end{center}
	
	\begin{theorem}[Barycenter convexity rigidity]\label{thm:main-intro}
		If \((E,\|\cdot\|,\mathcal L^n)\) is a finite-dimensional normed vector space with Lebesgue measure and satisfies \(\mathrm{BCD}(0,\infty)\), then \(\|\cdot\|\) is induced by an inner product.
	\end{theorem}

	Thus the passage from geodesics to barycenters is not a cosmetic
strengthening of convexity: it changes the class of infinitesimal models. Together with the stability of the barycenter curvature-dimension condition under measured blow-up, the following corollary rules out non-Hilbertian reversible Finsler tangent norms. Its proof is given in Section~\ref{sec:main-corollary-proofs}.
	
	\begin{corollary}[Exclusion of non-Riemannian reversible Finsler tangents]
		\label{cor:finsler-exclusion}
		Let \((M,F,\mathfrak m)\) be a smooth reversible Finsler manifold with a smooth positive
		reference measure. Assume that it satisfies \(\mathrm{BCD}(K,\infty)\). Then \(F_x\) is induced by an inner product on \(T_xM\) for every \(x\in M\).
		In particular, the Finsler structure is Riemannian.
	\end{corollary}

	In finite-dimensional normed spaces and in the corresponding Finsler setting, BCD detects the Riemannian case.
	Thus the barycenter curvature-dimension condition gives  a new synthetic Riemannian curvature-dimension condition: it implies the usual CD condition, but excludes the Finsler examples allowed by CD.
	
	Although RCD spaces satisfy BCD (cf. \cite{HanLiuZhuBCD}), the comparison with the RCD theory would be understood as a motivation rather than as an equivalence statement. In the RCD framework, the Riemannian character of the space is imposed through infinitesimal Hilbertianity. The present theorem points to another route: the Wasserstein Jensen's inequality already carries a Riemannianity constraint. This suggests studying a curvature-dimension theory in which Wasserstein barycenters, rather than only Wasserstein geodesics, are the testing objects. In that direction, one may ask which analytic and geometric consequences follow directly from the Wasserstein Jensen's inequality, and whether it yields stability, localization, rigidity, or regularity results that are not visible from two-point displacement convexity alone.
	
	\medskip
	
	The theorem also gives a Hilbertian characterization. Kakutani's theorem \cite{Kakutani1939} detects the inner product by requiring every two-dimensional section of the unit ball to be an ellipse; the Jordan--von Neumann theorem \cite{JordanvonNeumann1935} does so through the parallelogram law. For background on Hilbertian characterizations, see \cite{Amir1986}. The characterization here comes from the Wasserstein Jensen's inequality.
	
	\medskip
	
	\phantomsection
	\subsection*{Proof strategy}\label{subsec:proof-strategy}
	The proof has two main ingredients. In the strictly convex case, the argument works directly with calibrated rank-one transport branches; in the non-strictly convex case, it exploits the convex geometry of the unit ball.

	\medskip\noindent\textit{The rank-one polarization method.}
	For every covector \(r\), the square potential \(\alpha r(x)^2/2\) generates the affine rank-one map
	\[
	x\longmapsto x-\alpha r(x)J(r).
	\]
	For small \(\alpha\) this potential is \(c\)-concave by a one-dimensional dual-norm estimate. The four-branch identity
	\[
	\Phi_{p+q}+\Phi_{p-q}-2\Phi_p-2\Phi_q=0,
	\qquad \Phi_r(x)=\frac12r(x)^2,
	\]
	produces a zero-sum calibrated family with a unique Wasserstein barycenter. Strict convexity gives uniqueness, and the BCD condition can then be applied directly. Applying BCD to the family gives the parallelogram equality.
	
	\medskip\noindent\textit{The maximal-face method.} To prove strict convexity, we argue by contradiction. If the unit ball has a non-trivial flat face, then a supporting covector exposes a positive-dimensional face. In this face we choose two extremal translation anchors and a third face-valued perturbation. The anchors force the uniqueness of the Wasserstein barycenter, while the third branch gives a strict first-order entropy decrease through the Jacobian formula. This contradicts BCD and rules out flat faces.

	\medskip
	
	\phantomsection
	\subsection*{Organization.}\label{subsec:organization}
	Section~\ref{sec:preliminaries} collects preliminaries on norms, \(c\)-concavity, barycenters, calibrated maps, and entropy under changes of variables. Section~\ref{sec:strict-case} proves the Hilbertian conclusion under the additional assumption of strict convexity. Section~\ref{sec:strict-convexity} proves that BCD rules out flat faces of the unit ball.  Combining these two parts we  prove Theorem~\ref{thm:main-intro} and then prove Corollary~\ref{cor:finsler-exclusion}.

	\section{Preliminaries}\label{sec:preliminaries}
	
	\subsection{Norms, duality, and the Legendre map}\label{subsec:norms-duality}
	
	Let $E^*$ be the dual space and let $\|\cdot\|_*$ be the dual norm. We use basic Fenchel duality in the form standard in convex analysis \cite{Rockafellar1970}. Throughout the paper we set
	\[
	h(v)=\frac12\|v\|^2,
	\qquad c(x,y)=h(y-x).
	\]
	This only rescales the barycenter functional and does not affect barycenters. For the case \(K=0\), which is the only case used in the proof of the main theorem, it also leaves the Jensen inequality unchanged; for general \(K\) it only changes the normalization of the curvature parameter.
	For probability measures \(\mu,\nu\in\cP_2(E)\), we write
	\[
	\mathsf C(\mu,\nu):=\inf_{\pi\in\Pi(\mu,\nu)}
	\int_{E\times E} c(x,y)\,\dd\pi(x,y).
	\]
	Thus \(\mathsf C=\frac12 W_2^2\) for the 2-Wasserstein distance induced by \(\|\cdot\|\).
	
	The convex conjugate of $h$ is
	\[
	h^*(p)=\frac12\|p\|_*^2.
	\]
	If $\|\cdot\|$ is strictly convex, by Fenchel duality and the differentiability criterion for finite convex functions; see Rockafellar~\cite[Theorems 23.5 and 25.1]{Rockafellar1970}, $h^*$ is differentiable on $E^*$, and we write
	\[
	J(p):=\D h^*(p)\in E.
	\]
	The map $J$ is continuous, odd, and one-homogeneous:
	\[
	J(\lambda p)=\lambda J(p),\qquad \lambda\ge0,
	\qquad J(-p)=-J(p).
	\]
	Moreover,
	\[
	J(p)\in\partial h^*(p),
	\qquad p\in\partial h(J(p)).
	\]
	Conversely, if $p\in\partial h(v)$ and $h^*$ is differentiable at $p$, then $v=J(p)$.
	
	\subsection{$c$-concavity and optimal maps}\label{subsec:c-concavity}
	
	For a function $\psi:E\to\R\cup\{-\infty\}$, define the $c$-transform by
	\[
	\psi^c(y):=\inf_{x\in E}\{c(x,y)-\psi(x)\}.
	\]
	A function $\psi$ is $c$-concave if $\psi=\chi^c$ for some $\chi$. Equivalently, $\psi$ admits global $c$-supports: for every $x$ there exists $y$ such that
	\[
	\psi(z)\le c(z,y)-c(x,y)+\psi(x)
	\qquad\forall z\in E.
	\]
	The $c$-subdifferential $\partial^c\psi(x)$ is the set of such $y$.
	
	\begin{remark}
		For the cost $c(x,y)=h(y-x)$, $c$-concavity is not the same as usual concavity, except in special quadratic coordinates. In Kantorovich duality, a $c$-support at $x$ is a translated copy of the cost profile $h$, rather than an affine supporting hyperplane.
	\end{remark}
	
	If $\psi$ is differentiable at $x$ and $y\in\partial^c\psi(x)$, then
	\[
	\D\psi(x)\in -\partial h(y-x).
	\]
	If $\|\cdot\|$ is strictly convex, equivalently $h^*$ is differentiable, this gives the unique displacement
	\[
	y-x=J(-\D\psi(x)).
	\]
	
	\begin{lemma}[Negative convex functions are $c$-concave]
		\label{lem:negative-convex}
		Let $F:E\to\R$ be convex. Then $-F$ is $c$-concave.
	\end{lemma}
	
	\begin{proof}
		Fix $x_0$ and choose $p\in\partial F(x_0)$. Choose $v\in\partial h^*(-p)$, equivalently $-p\in\partial h(v)$, and set $y=x_0-v$. Convexity of $F$ gives
		\[
		F(z)-F(x_0)\ge p(z-x_0),
		\]
		and the subgradient inequality for $h$ gives
		\[
		h(z-y)-h(x_0-y)\ge -p(z-x_0).
		\]
		Combining these inequalities yields
		\[
		-F(z)\le h(z-y)-h(x_0-y)-F(x_0)
		=c(z,y)-c(x_0,y)-F(x_0).
		\]
		Thus $-F$ admits a global $c$-support at every $x_0$.
	\end{proof}

	We shall use the following dual calibration principle, whose idea goes back to \cite{CarlierEkeland2010}.
	
	\begin{lemma}[Zero-sum calibration]
		\label{lem:dual-calibration}
		Let \(E=\mathbb R^n\), let \(c(x,y)=h(y-x)\), and let
		\(\lambda_1,\ldots,\lambda_m>0\) satisfy \(\sum_i\lambda_i=1\).
		Let \(\psi_i:E\to\mathbb R\) be \(c\)-concave functions such that
		\[
		\sum_{i=1}^m \lambda_i\psi_i=0.
		\]
		Set \(\chi_i:=\psi_i^c\). Assume that all integrals appearing below are well-defined; this will be the
		case in all applications, where the potentials are affine or quadratic and the
		measures have finite second moments. Let \(\nu\in\mathcal P_2(E)\), and suppose that
		there are maps \(T_i:E\to E\) such that
		\[
		T_i(x)\in\partial^c\psi_i(x)
		\qquad\text{for \(\nu\)-a.e. }x.
		\]
		Define
		\[
		\mu_i:=(T_i)_\#\nu.
		\]
		Then \(\nu\) is a barycenter of \((\mu_i)_{i=1}^m\) with weights
		\((\lambda_i)_{i=1}^m\), namely
		\[
		\nu\in\operatorname*{argmin}_{\eta\in\mathcal P_2(E)}
		\sum_{i=1}^m\lambda_i \mathsf C(\eta,\mu_i).
		\]
		
	\end{lemma}
	
	\begin{proof}
		Since \(\chi_i=\psi_i^c\), we have
		\[
		\psi_i(x)+\chi_i(y)\le c(x,y)
		\qquad\forall x,y\in E.
		\]
		Therefore, for every \(\eta\in\mathcal P_2(E)\) and every
		\(\pi\in\Pi(\eta,\mu_i)\),
		\[
		\int c(x,y)\,\dd\pi(x,y)
		\ge
		\int\psi_i\,\dd\eta+\int\chi_i\,\dd\mu_i.
		\]
		Taking the infimum over \(\pi\) gives
		\[
		\mathsf C(\eta,\mu_i)
		\ge
		\int\psi_i\,\dd\eta+\int\chi_i\,\dd\mu_i.
		\]
		After multiplying by \(\lambda_i\) and summing in \(i\), we obtain
		\[
		\begin{aligned}
			\sum_{i=1}^m\lambda_i \mathsf C(\eta,\mu_i)
			&\ge
			\int\sum_{i=1}^m\lambda_i\psi_i\,\dd\eta
			+
			\sum_{i=1}^m\lambda_i\int\chi_i\,\dd\mu_i \\
			&=
			\sum_{i=1}^m\lambda_i\int\chi_i\,\dd\mu_i.
		\end{aligned}
		\]
		On the other hand, since
		\[
		T_i(x)\in\partial^c\psi_i(x),
		\]
		we have
		\[
		\psi_i(x)+\chi_i(T_i(x))=c(x,T_i(x))
		\qquad\text{for \(\nu\)-a.e. }x.
		\]
		Thus \((\id,T_i)_\#\nu\) is optimal from \(\nu\) to \(\mu_i\), and
		\[
		\mathsf C(\nu,\mu_i)
		=
		\int c(x,T_i(x))\,\dd\nu(x)
		=
		\int\psi_i\,\dd\nu+\int\chi_i\,\dd\mu_i.
		\]
		Averaging again and using \(\sum_i\lambda_i\psi_i=0\), we get
		\[
		\sum_{i=1}^m\lambda_i \mathsf C(\nu,\mu_i)
		=
		\sum_{i=1}^m\lambda_i\int\chi_i\,\dd\mu_i.
		\]
		Hence \(\nu\) attains the lower bound and is a barycenter.
	\end{proof}
	
	\subsection{A lemma in convex geometry}\label{subsec:convex-geometry-lemma}
	In Section~\ref{sec:strict-convexity}, we use this lemma to prove the uniqueness of the Wasserstein barycenter of a particular triple of measures $\{\mu_1,\mu_2,\mu_3\}$ in the non-strictly convex setting. The lemma and its proof are well known to experts, but we include them for completeness.
	\begin{lemma}\label{lem}
		Let \(F \subset \mathbb{R}^n\) be a nonempty compact convex set. Then there exists a nonzero linear functional
		\[
		\ell \in (\mathbb{R}^n)^*
		\]
		such that \(\ell\) attains both its maximum and its minimum on \(F\) at unique points.
	\end{lemma}
	
	\begin{proof}
		Define the support function of \(F\) by
		\[
		h_F(\ell):=\max_{x\in F}\ell(x),
		\qquad
		\ell\in (\mathbb{R}^n)^*.
		\]
		Since \(F\) is compact, \(h_F\) is finite everywhere. Moreover, since \(F\) is bounded, \(h_F\) is a convex Lipschitz function.
		
		We first observe that
		\[
		\partial h_F(\ell)
		=
		\argmax_{x\in F}\ell(x),
		\]
		where \(\mathbb{R}^n\) is identified with its bidual.
		
		Indeed, let \(x\in F\) satisfy
		\[
		\ell(x)=h_F(\ell).
		\]
		Then, for every \(m\in(\mathbb{R}^n)^*\),
		\[
		h_F(m)\geq m(x)
		=
		\ell(x)+(m-\ell)(x)
		=
		h_F(\ell)+(m-\ell)(x).
		\]
		Hence \(x\in\partial h_F(\ell)\).
		
		Conversely, suppose that \(x\in\partial h_F(\ell)\). Since \(h_F\) is the Fenchel conjugate of the indicator function of \(F\),
		\[
		h_F^*=\iota_F,
		\]
		the Fenchel equality implies that \(x\in F\). Taking \(m=0\) in the subgradient inequality gives
		\[
		0=h_F(0)
		\geq
		h_F(\ell)-\ell(x),
		\]
		and therefore
		\[
		\ell(x)\geq h_F(\ell).
		\]
		Since \(x\in F\), the reverse inequality is automatic, and thus
		\[
		\ell(x)=h_F(\ell).
		\]
		This proves the subdifferential identity.
		
		For a finite convex function on a finite-dimensional space, differentiability at a point is equivalent to the subdifferential at that point being a singleton. Therefore,
		\[
		h_F \text{ is differentiable at } \ell
		\quad\Longleftrightarrow\quad
		\argmax_{x\in F}\ell(x)
		\text{ is a singleton}.
		\]
		
		By the Rademacher theorem for Lipschitz functions, the complement of the set
		\[
		D_F
		:=
		\left\{
		\ell\in(\mathbb{R}^n)^*:
		h_F \text{ is differentiable at }\ell
		\right\}
		\]
		is a Lebesgue-null set. The same is true for \(-F\).
		
		Thus,
		\[
		D_F\cap D_{-F}
		\]
		contains a nonzero functional \(\ell\). For such an \(\ell\), the functional \(\ell\) has a unique maximizer on \(F\) and a unique minimizer on \(F\).
	\end{proof}

	\subsection{Jacobian formulas for entropy}\label{sec:entropy-jacobian}
	
	The entropy computations used below are based on finite-dimensional change-of-variables formulas. We state them here to avoid ambiguity with Wasserstein first variation formulas.
	
	\begin{lemma}[Entropy under a change of variables]
		\label{lem:entropy-change-of-variables}
		Let \(\nu=\rho\mathcal L^n\) be an absolutely continuous probability measure with finite entropy. Let \(S\) be a \(C^1\)-diffeomorphism from an open neighbourhood of \(\supp\nu\) onto its image, and assume
		\[
		\det \D S(x)>0
		\qquad\text{for all }x\in\supp\nu.
		\]
		Set \(\nu_S:=S_\#\nu\). Then
		\[
		\Ent(\nu_S)
		=\Ent(\nu)-\int_E \log\det\D S(x)\,\dd\nu(x).
		\]
	\end{lemma}
	
	\begin{proof}
		Let \(\rho_S\) be the density of \(\nu_S\). The Jacobian identity gives
		\[
		\rho_S(S(x))\det\D S(x)=\rho(x)
		\qquad\text{for }\nu\text{-a.e. }x.
		\]
		Changing variables in the entropy integral yields
		\[
		\begin{aligned}
			\Ent(\nu_S)
			&=\int_E \rho_S(S(x))\log\rho_S(S(x))\det\D S(x)\,\dd x \\
			&=\int_E \rho(x)\bigl(\log\rho(x)-\log\det\D S(x)\bigr)\,\dd x,
		\end{aligned}
		\]
		which is the claim.
	\end{proof}
	
	\begin{corollary}[Smooth perturbations]
		\label{cor:smooth-perturbation-entropy}
		Let \(\nu=\rho\mathcal L^n\) be as in Lemma~\ref{lem:entropy-change-of-variables}. Let
		\[
		V=v_0+Z,
		\qquad v_0\in E,\quad Z\in C_c^1(E;E).
		\]
		For \(|t|\) sufficiently small define \(S_t(x):=x+tV(x)\). Then
		\[
		\Ent((S_t)_\#\nu)
		=\Ent(\nu)-\int_E\log\det(I+t\D V(x))\,\dd\nu(x),
		\]
		and consequently
		\[
		\Ent((S_t)_\#\nu)
		=\Ent(\nu)-t\int_E \operatorname{div}V\,\dd\nu+O(t^2).
		\]
		Equivalently, if \(\dd\nu=\rho\,\dd x\), then
		\[
		\Ent((S_t)_\#\nu)
		=\Ent(\nu)-t\int_E \rho\,\operatorname{div}V\,\dd x+O(t^2).
		\]
	\end{corollary}
	
	\begin{proof}
		For \(|t|\) small, \(S_t\) is a \(C^1\)-diffeomorphism on a neighbourhood of \(\supp\nu\). Lemma~\ref{lem:entropy-change-of-variables} gives the exact logarithmic determinant formula. The expansion follows from
		\[
		\log\det(I+tA)=t\operatorname{tr}A+O(t^2\|A\|^2),
		\]
		uniformly for \(A=\D V(x)\), since \(\D V\) is bounded.
	\end{proof}
	
	\begin{corollary}[Affine maps]
		\label{cor:affine-entropy}
		Let \(A:E\to E\) be an invertible linear map with \(\det A>0\), and let \(b\in E\). Then, for every absolutely continuous probability measure \(\nu\) with finite entropy,
		\[
		\Ent((x\mapsto Ax+b)_\#\nu)
		=\Ent(\nu)-\log\det A.
		\]
	\end{corollary}

	\section{The strictly convex case}\label{sec:strict-case}
	
	In this section, assume that a finite-dimensional normed space \((E,\|\cdot\|,\mathcal L^n)\) satisfies \(\mathrm{BCD}(0,\infty)\) and that $\|\cdot\|$ is strictly convex. Strict convexity gives uniqueness for the Wasserstein barycenters used below, so the BCD condition can be applied directly.
	Strict convexity of the primal norm implies that
	\[
	h^*(p)=\frac12\|p\|_*^2
	\]
	is differentiable on \(E^*\). We write
	\[
	J:=\D h^*:E^*\to E.
	\]
	The purpose of this section is to prove that \(\|\cdot\|_*\) satisfies the parallelogram
	identity. This will imply that the dual norm, and hence also the original norm, is
	Hilbertian.
	
	The perturbations used below are rank-one affine perturbations generated by square functions
	\[
	x\longmapsto \frac12 r(x)^2,
	\qquad r\in E^*.
	\]
	Each branch is an extendable Wasserstein geodesic. The entropy computations are affine Jacobian computations covered by Corollary~\ref{cor:affine-entropy}.
	
	\subsection{Rank-one square potentials}\label{subsec:rank-one-square-potentials}
	
	For \(r\in E^*\), set
	\[
	m_r^2:=r(J(r))=\|r\|_*^2.
	\]
	The equality follows from the two-homogeneity of \(h^*\). If \(r=0\), then
	\(m_r=0\) and all formulas below are interpreted in the trivial sense.
	
	For \(\alpha\in\R\), define
	\[
	\Phi_{\alpha,r}(x):=\frac{\alpha}{2}\,r(x)^2,
	\]
	and define the affine rank-one map
	\[
	T_{\alpha,r}(x):=x-\alpha r(x)J(r).
	\]
	Thus \(T_{\alpha,r}=I-\alpha J(r)\otimes r\).
	Throughout this subsection we use the strict convexity assumption, so that
	\(h^*\) is differentiable on \(E^*\). 
	
	\begin{lemma}[Rank-one square potentials are calibrated]
		\label{lem:rank-one-square-calibration}
		Let \(r\in E^*\) and \(\alpha\in\R\). Assume
		\begin{equation}\label{eq:rank-one-admissibility}
			\alpha m_r^2<1.
		\end{equation}
		Then \(\Phi_{\alpha,r}\) is \(c\)-concave and
		\[
		T_{\alpha,r}(x)\in\partial^c\Phi_{\alpha,r}(x)
		\qquad\forall x\in E.
		\]
		Moreover the equality set is single-valued:
		\[
		\partial^c\Phi_{\alpha,r}(x)=\{T_{\alpha,r}(x)\}
		\qquad\forall x\in E.
		\]
	\end{lemma}
	
	\begin{proof}
		If \(r=0\), then \(\Phi_{\alpha,r}\equiv0\) and \(T_{\alpha,r}=\id\). Since
		\(h\ge0\) and \(h(v)=0\) only for \(v=0\), the assertion is immediate. We henceforth
		assume \(r\ne0\).
		
		Fix \(x\in E\) and put
		\[
		y:=T_{\alpha,r}(x)=x-\alpha r(x)J(r).
		\]
		We prove the global \(c\)-support inequality
		\begin{equation}\label{eq:rank-one-support}
			\Phi_{\alpha,r}(z)-\Phi_{\alpha,r}(x)
			\le h(y-z)-h(y-x)
			\qquad\forall z\in E.
		\end{equation}
		Write
		\[
		s:=r(x),
		\qquad
		u:=z-x,
		\qquad
		t:=r(u)=r(z-x).
		\]
		Then
		\begin{equation}\label{eq:rank-one-potential-diff}
			\Phi_{\alpha,r}(z)-\Phi_{\alpha,r}(x)
			=\frac{\alpha}{2}\bigl((s+t)^2-s^2\bigr)
			=\alpha st+\frac{\alpha}{2}t^2.
		\end{equation}
		On the other hand, by the dual norm inequality,
		\begin{equation}\label{eq:rank-one-dual-lower-bound}
			h(w)=\frac12\|w\|^2
			\ge \frac{r(w)^2}{2\|r\|_*^2}
			=\frac{r(w)^2}{2m_r^2}
			\qquad\forall w\in E.
		\end{equation}
		Since \(J(r)\in\partial h^*(r)\), Fenchel equality gives
		\[
		h(J(r))+h^*(r)=r(J(r))=m_r^2,
		\qquad h^*(r)=\frac12m_r^2,
		\]
		hence
		\begin{equation}\label{eq:h-Jr}
			h(J(r))=\frac12m_r^2.
		\end{equation}
		Using \eqref{eq:rank-one-dual-lower-bound} for \(w=y-z=-\alpha sJ(r)-u\), and
		using \eqref{eq:h-Jr} for \(y-x=-\alpha sJ(r)\), we obtain
		\[
		\begin{aligned}
			h(y-z)-h(y-x)
			&\ge
			\frac{r(-\alpha sJ(r)-u)^2}{2m_r^2}
			-\frac{\alpha^2s^2m_r^2}{2} \\
			&=
			\frac{(-\alpha sm_r^2-t)^2}{2m_r^2}
			-\frac{\alpha^2s^2m_r^2}{2} \\
			&=
			\alpha st+\frac{t^2}{2m_r^2}.
		\end{aligned}
		\]
		The assumption gives \(\alpha<m_r^{-2}\), hence in particular \(\alpha\le m_r^{-2}\).
		Therefore
		\[
		\alpha st+\frac{t^2}{2m_r^2}
		\ge
		\alpha st+\frac{\alpha}{2}t^2.
		\]
		Together with \eqref{eq:rank-one-potential-diff}, this proves
		\eqref{eq:rank-one-support}. Hence
		\(T_{\alpha,r}(x)\in\partial^c\Phi_{\alpha,r}(x)\) for every \(x\), and
		\(\Phi_{\alpha,r}\) is \(c\)-concave.
		
		It remains to prove uniqueness of the \(c\)-subgradient. Let
		\(z\in\partial^c\Phi_{\alpha,r}(x)\). Since \(\Phi_{\alpha,r}\) is differentiable,
		the first-order necessary condition for a \(c\)-support gives
		\[
		\D\Phi_{\alpha,r}(x)=\alpha r(x)r\in -\partial h(z-x).
		\]
		Equivalently,
		\[
		z-x\in \partial h^*(-\alpha r(x)r).
		\]
		Because \(h^*\) is differentiable, the right-hand side is the singleton
		\[
		\{J(-\alpha r(x)r)\}
		=\{-\alpha r(x)J(r)\}.
		\]
		Thus \(z=x-\alpha r(x)J(r)=T_{\alpha,r}(x)\). This proves the single-valued equality
		set.
	\end{proof}

	\begin{lemma}
		\label{lem:rank-one-extension}
		Let \(r\in E^*\), \(\alpha\in\R\), and assume \(\alpha m_r^2<1\). Then
		\(T_{\alpha,r}\) is invertible and
		\[
		\det T_{\alpha,r}=1-\alpha m_r^2>0.
		\]
	\end{lemma}
	
	\begin{proof}
		The determinant formula is the rank-one determinant identity
		\[
		\det(I-u\otimes r)=1-r(u),
		\]
		applied to \(u=\alpha J(r)\). The positivity follows from the assumption \(\alpha m_r^2<1\), so \(T_{\alpha,r}\) is invertible.
		
	\end{proof}
	
	\subsection{The polarization perturbation}\label{subsec:polarization-perturbation}
	
	For \(r\in E^*\), write
	\[
	\Phi_r(x):=\frac12 r(x)^2.
	\]
	We use the polarization identity
	\begin{equation}\label{eq:polarization-zero-sum}
		\Phi_{p+q}+\Phi_{p-q}-2\Phi_p-2\Phi_q=0
		\qquad\forall p,q\in E^*.
	\end{equation}
	Indeed, it is just
	\((p+q)^2+(p-q)^2=2p^2+2q^2\).
	
	Fix \(p,q\in E^*\). Choose \(\varepsilon>0\) so small that
	\begin{equation}\label{eq:epsilon-small-pq}
		\varepsilon M(p,q)<1,
		\qquad
		M(p,q):=\max\{\|p+q\|_*^2,\|p-q\|_*^2,2\|p\|_*^2,2\|q\|_*^2\},
	\end{equation}
	with the convention that there is no restriction if \(M(p,q)=0\). Define four potentials
	\begin{equation}\label{eq:polarization-potentials}
		\psi_1:=\varepsilon\Phi_{p+q},
		\qquad
		\psi_2:=\varepsilon\Phi_{p-q},
		\qquad
		\psi_3:=-2\varepsilon\Phi_p,
		\qquad
		\psi_4:=-2\varepsilon\Phi_q.
	\end{equation}
	By \eqref{eq:polarization-zero-sum},
	\begin{equation}\label{eq:four-zero-sum}
		\psi_1+\psi_2+\psi_3+\psi_4=0.
	\end{equation}
	By Lemma~\ref{lem:negative-convex} and Lemma~\ref{lem:rank-one-square-calibration}, each \(\psi_i\) is \(c\)-concave. The
	corresponding calibrated maps are
	\begin{equation}\label{eq:four-maps}
		\begin{aligned}
			T_1(x)&=x-\varepsilon(p+q)(x)J(p+q),\\
			T_2(x)&=x-\varepsilon(p-q)(x)J(p-q),\\
			T_3(x)&=x+2\varepsilon p(x)J(p),\\
			T_4(x)&=x+2\varepsilon q(x)J(q).
		\end{aligned}
	\end{equation}
	Their determinants are
	\begin{equation}\label{eq:four-determinants}
		\begin{aligned}
			d_1&:=\det T_1=1-\varepsilon\|p+q\|_*^2,\\
			d_2&:=\det T_2=1-\varepsilon\|p-q\|_*^2,\\
			d_3&:=\det T_3=1+2\varepsilon\|p\|_*^2,\\
			d_4&:=\det T_4=1+2\varepsilon\|q\|_*^2.
		\end{aligned}
	\end{equation}
	All four determinants are positive by \eqref{eq:epsilon-small-pq}.

	\begin{proposition}[Unique barycenter for the polarization family]
		\label{prop:polarization-unique-barycenter}
		Let \(\nu\) be a compactly supported absolutely continuous probability measure with finite
		entropy, and set
		\[
		\mu_i:=(T_i)_\#\nu,
		\qquad i=1,2,3,4,
		\]
		where the \(T_i\) are the four maps defined in
		\eqref{eq:four-maps}. Then \(\nu\) is the
		unique Wasserstein barycenter of
		\((\mu_1,\mu_2,\mu_3,\mu_4)\) with equal weights \(1/4\).
	\end{proposition}
	
	\begin{proof}
		By Lemma~\ref{lem:dual-calibration}, we know that \(\nu\) is a Wasserstein barycenter. 
		
		We now prove uniqueness. Let \(\eta\) be any Wasserstein barycenter of
		\((\mu_1,\mu_2,\mu_3,\mu_4)\) with equal weights. It follows that 
		\(\eta\) also minimizes the barycenter functional. Combining this with \eqref{eq:four-zero-sum}, we obtain
		
		\begin{equation}\label{eq:polarization-barycenter-lower-bound}
			\frac14\sum_{i=1}^4 \mathsf W_2^2(\eta,\mu_i)
			=
			\frac14\sum_{i=1}^4 \mathsf W_2^2(\nu,\mu_i)
			=
			\frac14\sum_{i=1}^4\int_E\psi_i^c\,\dd\mu_i.
		\end{equation}
		
		For each \(i\), define the nonnegative Kantorovich gap
		\[
		G_i(x,y):=\|x-y\|^2-\psi_i(x)-\psi_i^c(y)\ge0.
		\]
		Let \(\pi_i\) be an optimal plan from \(\eta\) to \(\mu_i\). Equality in \eqref{eq:polarization-barycenter-lower-bound} implies
		\[
		\frac14\sum_{i=1}^4\int_{E\times E}G_i(x,y)\,\dd\pi_i(x,y)=0.
		\]
		Because each term is nonnegative, we have
		\[
		G_i(x,y)=0
		\qquad \pi_i\text{-a.e.}
		\]
		for every \(i\). Hence
		\[
		y\in\partial^c\psi_i(x)
		\qquad \pi_i\text{-a.e.}
		\]
		By Lemma~\ref{lem:rank-one-square-calibration}, the equality set of each
		\(\psi_i\) is single-valued:
		\[
		\partial^c\psi_i(x)=\{T_i(x)\}
		\qquad\forall x\in E.
		\]
		Therefore
		\[
		\pi_i=(\id,T_i)_\#\eta,
		\qquad\text{and hence}\qquad
		\mu_i=(T_i)_\#\eta.
		\]
		In particular, for \(i=1\),
		\[
		(T_1)_\#\eta=\mu_1=(T_1)_\#\nu.
		\]
		By Lemma~\ref{lem:rank-one-extension}, \(T_1\) is an invertible affine map.
		Pushing forward by \(T_1^{-1}\), we obtain
		\[
		\eta=\nu.
		\]
		Thus \(\nu\) is the unique barycenter.
	\end{proof}
	
	\subsection{The entropy inequality gives the parallelogram law}\label{subsec:entropy-parallelogram}
	
	\begin{proposition}[BCD forces the dual parallelogram identity]
		\label{prop:dual-parallelogram}
		Assume \((E,\|\cdot\|,\mathcal L^n)\) satisfies \(\mathrm{BCD}(0,\infty)\), and assume the norm is strictly convex. Then
		\begin{equation}\label{eq:dual-parallelogram}
			\|p+q\|_*^2+\|p-q\|_*^2
			=2\|p\|_*^2+2\|q\|_*^2
			\qquad\forall p,q\in E^*.
		\end{equation}
	\end{proposition}
	
	\begin{proof}
		Fix \(p,q\in E^*\), and choose \(\varepsilon>0\) satisfying
		\eqref{eq:epsilon-small-pq}. Let \(\nu\) be any compactly supported absolutely continuous
		probability measure with finite entropy, and construct \(\mu_1,\dots,\mu_4\) from
		\eqref{eq:four-maps}. By Proposition~\ref{prop:polarization-unique-barycenter}, the unique
		barycenter is \(\nu\). Hence BCD gives
		\begin{equation}\label{eq:bcd-four-entropy}
			\Ent(\nu)
			\le \frac14\sum_{i=1}^4\Ent(\mu_i).
		\end{equation}
		Each \(T_i\) is affine, so Corollary~\ref{cor:affine-entropy} gives
		\[
		\Ent(\mu_i)=\Ent(\nu)-\log d_i,
		\]
		where the determinants \(d_i\) are listed in \eqref{eq:four-determinants}. Substituting
		this into \eqref{eq:bcd-four-entropy} yields
		\begin{equation}\label{eq:product-le-one-first}
			d_1d_2d_3d_4\le1.
		\end{equation}
		In other words,
		\begin{equation}\label{eq:first-product-expanded}
			\begin{aligned}
				&(1-\varepsilon\|p+q\|_*^2)
				(1-\varepsilon\|p-q\|_*^2)
				(1+2\varepsilon\|p\|_*^2)
				(1+2\varepsilon\|q\|_*^2)
				\le1.
			\end{aligned}
		\end{equation}
		Letting \(\varepsilon\downarrow0\) gives the first-order inequality
		\begin{equation}\label{eq:first-parallelogram-side}
			\|p+q\|_*^2+\|p-q\|_*^2
			\ge 2\|p\|_*^2+2\|q\|_*^2.
		\end{equation}
		
		Since \eqref{eq:first-parallelogram-side} holds for every pair of covectors, we may
		apply it to the pair \(p+q,p-q\). This gives
		\begin{equation}\label{eq:second-parallelogram-side}
			\|p+q\|_*^2+\|p-q\|_*^2
			\le 2\|p\|_*^2+2\|q\|_*^2.
		\end{equation}
		Combining \eqref{eq:first-parallelogram-side} and \eqref{eq:second-parallelogram-side}
		proves \eqref{eq:dual-parallelogram}.
	\end{proof}

	\begin{corollary}[Dual parallelogram identity implies Hilbertianity]
		\label{cor:hilbertianity}
		If \(\|\cdot\|_*\) satisfies \eqref{eq:dual-parallelogram}, then \(\|\cdot\|\) is induced
		by an inner product on \(E\).
	\end{corollary}
	
	\begin{proof}
		By the Jordan--von Neumann theorem, the dual norm $\|\cdot\|_*$ is induced by an inner product. Hence its dual norm $\|\cdot\|_{**}$ is also induced by an inner product. Since $\|\cdot\|=\|\cdot\|_{**}$, the result follows.
	\end{proof}

	\section{Strict convexity from BCD}\label{sec:strict-convexity}
	
		It remains to consider the case where the norm is not strictly convex. In this section, we show that $\mathrm{BCD}(0,\infty)$ in fact forces strict convexity. The main difficulty in the non-strictly convex case is that Wasserstein barycenters of general distributions $\sum_{i=1}^{m}\lambda_i\delta_{\mu_i}$ are typically non-unique, so the $\mathrm{BCD}$ condition cannot be applied directly. We overcome this difficulty by constructing a special distribution below whose Wasserstein barycenter is unique, using the flat faces of the unit ball. Recall that the presence of flat faces in the unit ball is precisely the obstruction to strict convexity.
	
	\begin{proposition}[Maximal-face trapping: barycenter convexity rules out flat faces]
		\label{prop:strict-convexity}
		If $(E,\|\cdot\|,\cL^n)$ satisfies $\mathrm{BCD}(0,\infty)$, then the unit ball of $\|\cdot\|$ has no non-trivial line segment. Hence $\|\cdot\|$ is strictly convex.
	\end{proposition}

	\begin{proof}
	
	\emph{Step 1: Characterization of the exposed face.} 
	
	Assume that the unit ball $B=\{v:\|v\|\le1\}$ is not strictly convex. Then there exist distinct $u_0,u_1\in\partial B$ with $[u_0,u_1]\subset\partial B$. Let $m=(u_0+u_1)/2$ and choose a supporting covector $p\in E^*$ such that
		\[
		\|p\|_*=1,
		\qquad p(m)=1,
		\qquad p(v)\le1\quad\forall v\in B.
		\]
		Then $p(u_0)=p(u_1)=1$, and
		\[
		F:=\{v\in B:p(v)=1\}
		\]
		is a positive-dimensional exposed face.
		
		We first prove the face identity used throughout the proof. Since $h^*(p)=\frac12\|p\|_*^2=\frac12$, Fenchel equality gives
		\[
		v\in\partial h^*(p)
		\quad\Longleftrightarrow\quad
		h(v)+h^*(p)=p(v).
		\]
		Writing $r=\|v\|$, the right-hand side implies
		\[
		\frac12 r^2+\frac12=p(v)\le \|p\|_*\|v\|=r.
		\]
		Thus $(r-1)^2\le0$, so $r=1$ and $p(v)=1$, i.e. $v\in F$. Conversely, if $v\in F$, then $\|v\|=1$ and $p(v)=1$, so the Fenchel equality holds. Hence
		\[
		\partial h^*(p)=F.
		\]
		By the two-homogeneity of $h^*$, for every $\alpha\ge0$,
		\[
		\partial h^*(\alpha p)=\alpha F,
		\]
		with the convention $0F=\{0\}$.
		
		\medskip
	\noindent	\emph{Step 2: A face-valued perturbation with negative averaged divergence.}
		
		Let $T=\operatorname{span}(F-F)$. Choose a linear functional $\ell$ on $E$ whose restriction to $F$ has a unique maximizer $a$ and a unique minimizer $b$:
		\[
		a=\argmax_{f\in F}\ell(f),
		\qquad b=\argmin_{f\in F}\ell(f).
		\]
		Such an $\ell$ exists by Lemma~\ref{lem}.
		
		All directional derivatives and divergences below are computed with respect to a fixed linear coordinate system inducing the Lebesgue measure \(\mathcal L^n\).
		Pick $e_0\in\ri F$ and $0\ne\tau\in T$. Choose a smooth compactly supported probability density \(\rho\) such that \(D_\tau\rho\) is not identically zero, where \(D_\tau\) denotes the directional derivative in the direction \(\tau\). Set
		\[
		\varphi:=D_\tau\rho.
		\]
		Then \(\varphi\in C_c^\infty(E)\), and integration by parts gives
		\[
		\int_E \rho\,D_\tau\varphi\,\dd x
		=-\int_E (D_\tau\rho)^2\,\dd x<0.
		\]
		For sufficiently small $\delta>0$, define
		\[
		e(x):=e_0+\delta\varphi(x)\tau
		\]
		so that $e(x)\in F$ for every $x$. Indeed, \(e_0\in\ri F\) and
		\(\tau\in\operatorname{span}(F-F)\), so \(e_0+s\tau\in F\) for all sufficiently small
		\(|s|\). Since \(\diver e=\delta D_\tau\varphi\), we have
		\begin{equation}\label{eq:face-negative-divergence}
			\int_E \rho\,\diver e\,\dd x<0.
		\end{equation}
		
		\medskip
		\noindent\emph{Step 3: Three calibrated branches and uniqueness of the barycenter.}
		
		Let $\nu=\rho\cL^n$. For small $t>0$, define
		\[
		T_1^t(x)=x+ta,
		\qquad T_2^t(x)=x+tb,
		\qquad T_3^t(x)=x-2t e(x),
		\]
		and set $\mu_i^t=(T_i^t)_\#\nu$. Since \(\D e\) is bounded, for sufficiently small \(t\) the map \(T_3^t=\id-2t e\) is a \(C^1\)-diffeomorphism on an open neighbourhood of \(\supp\nu\). Thus all $\mu_i^t$ have finite entropy and compact support.
		
		Set
		\[
		q_1=p,
		\qquad q_2=p,
		\qquad q_3=-2p,
		\]
		and
		\[
		\psi_i^t(y):=-t\,q_i(y).
		\]
		Then $\psi_1^t+\psi_2^t+\psi_3^t=0$. For a general covector $q\in E^*$,
		\[
		\begin{aligned}
			(\psi_q^t)^c(z)
			&=\inf_y\{h(z-y)+tq(y)\} \\
			&=tq(z)-h^*(tq).
		\end{aligned}
		\]
		Equality holds if and only if
		\[
		z-y\in\partial h^*(tq),
		\]
		again by Fenchel equality. Therefore
		\[
		T_1^t(x)-x=ta\in tF=t\partial h^*(p),
		\]
		\[
		T_2^t(x)-x=tb\in tF=t\partial h^*(p),
		\]
		and
		\[
		T_3^t(x)-x=-2t e(x)\in -2tF=t\partial h^*(-2p).
		\]
		Thus the plans $(\id,T_i^t)_\#\nu$ are calibrated by $\psi_i^t$ and their $c$-transforms. Lemma~\ref{lem:dual-calibration} implies that $\nu$ is a Wasserstein barycenter of $\mu_1^t,\mu_2^t,\mu_3^t$ with equal weights.
		
		We now prove uniqueness. Let $\eta$ be any Wasserstein barycenter, and denote by $\chi_i^t=(\psi_i^t)^c$ the dual partners. For each \(i\), let \(\pi_i\) be an optimal plan from \(\eta\) to \(\mu_i^t\), and set
		\[
		G_i^t(y,z):=c(y,z)-\psi_i^t(y)-\chi_i^t(z)\ge0.
		\]
		Since \(\psi_1^t+\psi_2^t+\psi_3^t=0\), we have
		\[
		\frac13\sum_{i=1}^3\int G_i^t\,\dd\pi_i
		=
		\frac13\sum_{i=1}^3 \mathsf C(\eta,\mu_i^t)
		-
		\frac13\sum_{i=1}^3\int \chi_i^t\,\dd\mu_i^t.
		\]
		The second term is the barycenter dual value, already attained at \(\nu\). Since
		\(\eta\) is also a Wasserstein barycenter, the right-hand side is zero. Hence
		\[
		\frac13\sum_{i=1}^3\int G_i^t\,\dd\pi_i=0.
		\]
		Because every \(G_i^t\ge0\), it follows that
		\(G_i^t=0\) for \(\pi_i\)-almost every point and every \(i\). Consequently, for $i=1,2$ the optimal plan $\pi_i\in\Pi(\eta,\mu_i^t)$ is supported on the corresponding equality set,
		\[
		z-y\in tF.
		\]
		
		Consider $i=1$. Let $S_a(x)=x+ta$. Since $\mu_1^t=(S_a)_\#\nu$, push $\pi_1$ forward by the map
		\[
		(y,z)\longmapsto (y,x),\qquad x=S_a^{-1}(z)=z-ta.
		\]
		The resulting plan $\widehat\pi_1$ has marginals $\eta$ and $\nu$. On its support, $z-y=tf$ for some $f\in F$, and $z=x+ta$; hence
		\[
		y=x+t(a-f).
		\]
		Therefore
		\[
		\ell(y)-\ell(x)=t\bigl(\ell(a)-\ell(f)\bigr)\ge0.
		\]
		Integrating with respect to $\widehat\pi_1$ gives
		\begin{equation}\label{eq:anchor-a-ineq}
			\int \ell\,\dd\eta-\int \ell\,\dd\nu
			=\int t\bigl(\ell(a)-\ell(f)\bigr)\,\dd\widehat\pi_1\ge0.
		\end{equation}
		
		For $i=2$, with $S_b(x)=x+tb$ and $x=S_b^{-1}(z)=z-tb$, the same argument gives a plan $\widehat\pi_2$ with marginals $\eta$ and $\nu$ and
		\[
		y=x+t(b-f),
		\qquad f\in F.
		\]
		Thus
		\[
		\ell(y)-\ell(x)=t\bigl(\ell(b)-\ell(f)\bigr)\le0,
		\]
		and hence
		\begin{equation}\label{eq:anchor-b-ineq}
			\int \ell\,\dd\eta\le\int \ell\,\dd\nu.
		\end{equation}
		Combining \eqref{eq:anchor-a-ineq} and \eqref{eq:anchor-b-ineq}, equality holds in \eqref{eq:anchor-a-ineq}. The integrand $t(\ell(a)-\ell(f))$ is nonnegative, so it vanishes \(\widehat\pi_1\)-almost everywhere. Since $a$ is the unique maximizer of $\ell$ on $F$, we have $f=a$ \(\widehat\pi_1\)-almost everywhere. Hence $y=x$ \(\widehat\pi_1\)-almost everywhere. The two marginals of $\widehat\pi_1$ are therefore equal, and $\eta=\nu$.
		Thus the barycenter is unique.

		\medskip
		\noindent\emph{Step 4: The entropy contradiction.}
		
		 Applying $\mathrm{BCD}(0,\infty)$ gives
		\[
		\Ent(\nu)
		\le\frac13\Ent(\mu_1^t)+\frac13\Ent(\mu_2^t)+\frac13\Ent(\mu_3^t).
		\]
		The first two measures are translations of $\nu$, hence have the same entropy. Therefore
		\begin{equation}\label{eq:face-bcd-entropy}
			\Ent(\nu)\le \Ent(\mu_3^t).
		\end{equation}
		But $\mu_3^t=(\id-2t e)_\#\nu$. Applying Corollary~\ref{cor:smooth-perturbation-entropy} with $V=-2e$ gives
		\[
		\Ent(\mu_3^t)
		=\Ent(\nu)+2t\int_E\rho\,\diver e\,\dd x+O(t^2).
		\]
		By \eqref{eq:face-negative-divergence}, this is strictly smaller than $\Ent(\nu)$ for small $t>0$, contradicting \eqref{eq:face-bcd-entropy}. Hence no positive-dimensional exposed face exists, and the unit ball is strictly convex.
	\end{proof}
	
	\subsection*{Proofs of the main theorem and the Finsler corollary}\label{sec:main-corollary-proofs}
	
	\begin{proof}[Proof of Theorem~\ref{thm:main-intro}]
		Proposition~\ref{prop:strict-convexity} implies that the norm is strictly convex. Consequently \(h^*\) is differentiable and the rank-one
		square perturbations constructed above are available.
		
		Proposition~\ref{prop:dual-parallelogram} gives the
		parallelogram identity. Corollary~\ref{cor:hilbertianity} then implies that the original
		norm is induced by an inner product. This proves the theorem.
	\end{proof}
	
	\begin{proof}[Proof of Corollary~\ref{cor:finsler-exclusion}]
		Fix \(x\in M\). For \(r>0\), consider the rescaled pointed metric-measure space
		\[
		(M,r^{-1}\mathsf d_F,c_r\mathfrak m,x),
		\]
		where \(c_r>0\) is a normalizing constant chosen for the measured blow-up. For the
		rescaled distance \(r^{-1}\mathsf d_F\), the squared Wasserstein distances are
		multiplied by \(r^{-2}\). Hence the original \(\mathrm{BCD}(K,\infty)\) inequality
		becomes a \(\mathrm{BCD}(K r^2,\infty)\) inequality for the rescaled space. Multiplying
		the reference measure by \(c_r>0\) only adds the same constant to all entropy terms and
		therefore does not affect the Jensen inequality.
		
		By the standard first-order blow-up of a smooth reversible Finsler manifold, the pointed
		measured spaces above converge, as \(r\downarrow0\), to
		\[
		(T_xM,F_x,a_x\mathcal L_x,0),
		\]
		where \(a_x>0\) and \(\mathcal L_x\) is a Lebesgue measure on \(T_xM\); see
		\cite[Chapter~1]{BaoChernShen} for the Finsler tangent norm and \cite{OhtaSturm} for the
		weighted Finsler metric-measure setting. Since \(K r^2\to0\), the stability theorem for
		the barycenter curvature-dimension condition under measured Gromov--Hausdorff convergence
		\cite[Theorem~6.6]{HanLiuZhuBCD} gives the \(\mathrm{BCD}(0,\infty)\) Jensen inequality
		on the tangent space, at least for compactly supported families of measures. This is the
		form used in the proof of Theorem~\ref{thm:main-intro}.
		
		It follows that \(F_x\) is induced by an inner product. Since
		\(x\) is arbitrary, \(F\) is Riemannian.
	\end{proof}

	\bigskip
	
	\phantomsection
	\section*{Funding}\label{sec:funding}
	This work was supported in part by the National Key R\&D Programs of China (2021YFA1000900, 2021YFA1002200), the National Natural Science Foundation of China (12201596), the Shandong Provincial Natural Science Foundation (ZR2025QB05), and the Taishan Scholars Program of Shandong Province (tsqn202408059).
	
	\medskip
	
	\phantomsection
	\section*{Acknowledgements}\label{sec:acknowledgements}
	The authors would like to thank Emanuel Milman for valuable discussions concerning the relation between \(\mathrm{BCD}\) and the classical curvature-dimension theory.

	\medskip
	
	\phantomsection
	\section*{Declarations}\label{sec:declarations}
	The authors declare that they have no conflict of interest and that the manuscript has no associated data.
\end{document}